\documentclass[a4paper,11pt]{article}%
\usepackage{amsmath}%
\usepackage{amsfonts}%
\usepackage{amssymb}%
\usepackage{graphicx}
\newtheorem{theorem}{Theorem}

\newtheorem{definition}[theorem]{Definition}

\newtheorem{lemma}[theorem]{Lemma}

\newtheorem{proposition}[theorem]{Proposition}

\newenvironment{proof}[1][Proof]{\textbf{#1.} }{\ \rule{0.5em}{0.5em}}
\renewcommand{\footnote}{\endnote}
\begin{document}

\title{TENSOR COMMUTATION MATRICES AND SOME GENERALIZATIONS OF THE PAULI MATRICES}
\author{Christian RAKOTONIRINA
\\Institut Sup\'erieur de Technologie d'Antananarivo, 
\\IST-T, BP 8122, Madagascar
\\ e-mail: rakotopierre@refer.mg
\\ Joseph RAKOTONDRAMAVONIRINA
\\D\'epartement de Math\'ematiques et Informatique, 
\\Universit\'e d'Antananarivo, Madagascar
\\ e-mail: joseph.rakotondramavonirina@univ-antananarivo.mg}

\maketitle

\begin{abstract}
In this paper, some tensor commutation matrices are expressed in termes of the generalized Pauli matrices by tensor products of the Pauli matrices.
\end{abstract}

\section{Introduction}

The tensor product of matrices is not commutative in general.However, a tensor commutation  matrix (TCM) ${n\otimes p}$, $S_{n\otimes p}$ commutes the tensor product $A\otimes B$ for any $A \in  \textbf{C}^{n\times n} $ and $B \in  \textbf{C}^{p\times p}$ as the following 
\begin{equation*} 
S_{n\otimes p}(A\otimes B)=(B\otimes A) S_{n\otimes p}
\end{equation*}
The tensor commutation matrices (TCMs) are useful in quantum theory and for solving matrix equations. In quantum theory $S_{2\otimes 2}$ can be expressed in the 
following way (Cf. for example\cite{Faddev95, Verstraete02, Fujii01})
\begin{equation}\label{eq1}
S_{2\otimes 2} =\frac{1}{2}  I_2 \otimes I_2+ \frac{1}{2}\sum_{i=1}^3 \sigma_i \otimes \sigma_i 
\end{equation}
where 
$\sigma_1=\begin{pmatrix}
          0 & 1\\
          1 & 0
         \end{pmatrix}$, 
$\sigma_2=\begin{pmatrix}
          0 & -i\\
          i &  0
         \end{pmatrix}$ and
$\sigma_3=\begin{pmatrix}
          1 & 0\\
          0 & -1
         \end{pmatrix}$
are the Pauli matrices, $I_2$ is the ${2\times2}$ unit-matrix.\\

\begin{equation*}
S_{2 \otimes 2} =\begin{pmatrix}
                1 & 0 & 0 & 0\\
                0 & 0 & 1 & 0\\
                0 & 1 & 0 & 0\\
                0 & 0 & 0 & 1
               \end{pmatrix}                 
\end{equation*}

\begin{equation*}
S_{3 \otimes 3} = \begin{pmatrix}
                 1 & 0 & 0 & 0 & 0 & 0 & 0 & 0 & 0\\
                 0 & 0 & 0 & 1 & 0 & 0 & 0 & 0 & 0\\
                 0 & 0 & 0 & 0 & 0 & 0 & 1 & 0 & 0\\
                 0 & 1 & 0 & 0 & 0 & 0 & 0 & 0 & 0\\
                 0 & 0 & 0 & 0 & 1 & 0 & 0 & 0 & 0\\
                 0 & 0 & 0 & 0 & 0 & 0 & 0 & 1 & 0\\
                 0 & 0 & 1 & 0 & 0 & 0 & 0 & 0 & 0\\
                 0 & 0 & 0 & 0 & 0 & 1 & 0 & 0 & 0\\
                 0 & 0 & 0 & 0 & 0 & 0 & 0 & 0 & 1
                  \end{pmatrix}               
\end{equation*}
The Gell-Mann matrices are a generalization of the Pauli matrices. The TCM ${n\otimes n}$ can be expressed in terms of ${n\times n}$ Gell-Mann matrices,
under the following expression \cite{Rakotonirina08}
\begin{equation}\label{eq2}
S_{n\otimes n} =\frac{1}{n}  I_{n} \otimes I_{n}+ \frac{1}{2}\sum_{i=1}^{n^{2}-1}\Lambda_i \otimes \Lambda_i 
\end{equation}
This expression of $S_{n\otimes n}$ suggests us the topic of generalizing the formula (\ref{eq1}) to an expression in terms of some generalized Pauli matrices.\\
For the calculus, we have used SCILAB a mathematical software for numerical analysis.
 
\section{SOME GENERALIZATIONS OF THE PAULI MATRICES}
In this section we give some generalizations of the Pauli matrices other than the Gell-Mann matrices
\subsection{KIBLER MATRICES}
The Kibler matrices are\\ 
$ k_1=\begin{pmatrix}
          1 & 0 & 0\\
          0 & 1 & 0\\
          0 & 0 & 1
         \end{pmatrix}$, 
$ k_2=\begin{pmatrix}
          0 & 1 & 0\\
          0 & 0 & 1\\
          1 & 0 & 0
         \end{pmatrix}$ 
$ k_3=\begin{pmatrix}
         0 & 0 & 1\\
         1 & 0 & 0\\
         0 & 1 & 0
         \end{pmatrix}$,\\
$ k_4=\begin{pmatrix}
          1 & 0 & 0\\
          0 & q & 0\\
          0 & 0 & q^{2} 
         \end{pmatrix}$, 
$ k_5=\begin{pmatrix}
          1 & 0 & 0\\
          0 & q^{2} & 0\\
          0 & 0 & q
         \end{pmatrix}$, 
$ k_6=\begin{pmatrix}
         0 & q & 0\\
         0 & 0 & q^{2}\\
         1 & 0 & 0
         \end{pmatrix}$,\\
$ k_7=\begin{pmatrix}
          0 & 0 & q\\
          1 & 0 & 0\\
          0 & q^{2} &0 
         \end{pmatrix}$, 
$ k_8=\begin{pmatrix}
          0 & 0 & q^{2} \\
          1 & 0 & 0\\
          0 & q & 0
         \end{pmatrix}$ 
$ k_9=\begin{pmatrix}
         0 & q^{2} & 0\\
         0 & 0 & q\\
         1 & 0 & 0
         \end{pmatrix}$,
are the  ${3\times 3}$ Kibler matrices \cite{kibler09}.\\
The Kibler matrices are traceless and
\begin{equation}\label{eq3}
\frac{1}{3}  I_{3}\otimes I_{3 }+ \frac{1}{3}\sum_{i=1}^8 \ k_i \otimes \ k_i =P
\end{equation}
with 
$ P =\begin{pmatrix}
                 0 & 0 & 0 & 0 & 1 & 0 & 0 & 0 & 0\\
                 0 & 1 & 0 & 0 & 0 & 0 & 0 & 0 & 0\\
                 0 & 0 & 0 & 0 & 0 & 0 & 0 & 1 & 0\\
                 0 & 0 & 0 & 1 & 0 & 0 & 0 & 0 & 0\\
                 1 & 0 & 0 & 0 & 0 & 0 & 0 & 0 & 0\\
                 0 & 0 & 0 & 0 & 0 & 0 & 1 & 0 & 0\\
                 0 & 0 & 0 & 0 & 0 & 1 & 0 & 0 & 0\\
                 0 & 0 & 1 & 0 & 0 & 0 & 0 & 0 & 0\\
                 0 & 0 & 0 & 0 & 0 & 0 & 0 & 0 & 1
       
      \end{pmatrix}$
a permutation matrix.
\subsection{THE NONIONS}
The nonions matrices are \cite{volkov10}\\
 $ q_0=\begin{pmatrix}
          1 & 0 & 0\\
          0 & 1 & 0\\
          0 & 0 & 1
         \end{pmatrix}$
$ q_1=\begin{pmatrix}
          0 & 1 & 0\\
          0 & 0 & 1\\
          1 & 0 & 0
         \end{pmatrix}$, 
$ q_2=\begin{pmatrix}
          0 & 1 & 0 \\
          0 & 0 & j \\
          j^{2} & 0 & 0
         \end{pmatrix}$\\ 
$ q_3=\begin{pmatrix}
         0 & 1 & 0\\
         0 & 0 & j^{2}\\
         j & 0 & 0
         \end{pmatrix}$,
$ q_4=\begin{pmatrix}
         0 & 0 & 1\\
         1 & 0 & 0\\
         0 & 1 & 0
         \end{pmatrix}$,
$ q_5=\begin{pmatrix}
          0 & 0 & j\\
          1 & 0 & 0\\
          0 &j^{2} & 0
         \end{pmatrix}$,\\
$ q_6=\begin{pmatrix}
          0 & 0 & j^{2} \\
          1 & 0 & 0     \\
          0 & j & 0
         \end{pmatrix}$, 
$ q_7=\begin{pmatrix}
          j & 0 & 0\\
          0 & j^{2} & 0\\
          0 & 0 & 1
         \end{pmatrix}$, 
$ q_8=\begin{pmatrix}
          j^{2} & 0 & 0\\
          0 & j & 0\\
          0 & 0 & 1
         \end{pmatrix}$  

The nonions are traceless and
\begin{equation}\label{eq4}
\frac{1}{3}  I_{3}\otimes I_{3 }+ \frac{1}{3}\sum_{i=1}^8 \ q_i \otimes \ q_i = P
\end{equation}

\subsection{GENERALIZATION BY TENSOR PRODUCTS OF PAULI MATRICES}
There are also some generalization of the Pauli matrices constructed by tensor products of the Pauli matrices, namely.
$(\sigma_i\otimes \sigma_j)_{0\leq i,j \leq3 }$,$(\sigma_i\otimes \sigma_j \otimes \sigma_k)_{0\leq i,j,k \leq 3}$, $(\sigma_{i_1}\otimes 
\sigma_{i_2} \otimes\ldots\otimes \sigma_{i_n})_{0\leq i_1,i_2,\ldots,i_n \leq 3}$ (Cf. for example, \cite{rigetti04,saniga06}).
The elements of the set $(\sigma_{i_1}\otimes\sigma_{i_2} \otimes\ldots\otimes \sigma_{i_n})_{0\leq i_1,i_2,\ldots,i_n \leq 3}$ satisfy the following properties
(Cf. for example, \cite{rigetti04,saniga06})
\begin{equation}\label{e5}
\Sigma_j^+=\Sigma_j \; (hermitian)
\end{equation}
\begin{equation}\label{e6}
\Sigma_j^2=I_{2^n}\;(Square\; root\; of\; unity )
\end{equation}
\begin{equation}\label{e7}
Tr\Sigma_j^{+}\Sigma_k=2^{n}\delta_{ij}\; (Orthogonal)
\end{equation}
\section{EXPRESSING A TENSOR COMMUTATION MATRIX IN TERME OF THE GENERALIZED PAULI MATRICES}
\begin{definition}
For $n\in N^{*},n\geq2$, we call tensor commutation matrix $n\otimes n$ the permutation matrix $ S_{n\otimes n}$ such that
\begin{equation*}
S_{n\otimes n} (a\otimes b)= \label{eq1}b\otimes a 
\end{equation*}
for any  $a$,$b \in \textbf{C}^{n\times 1}$.
\end{definition}
The relations (\ref{eq1}), (\ref{eq2}), (\ref{eq3}) and (\ref{eq4}) suggest us that there should be a generalization of the Pauli matrices $(s_i)_{0\leq i \leq n^2 -1}$ such that 
\begin{equation}\label{eq5}
S_{n\otimes n} =\frac{1}{n}  I_{n} \otimes I_{n}+ \frac{1}{n}\sum_{i=1}^{n^{2}-1}\ s_i \otimes \ s_i 
\end{equation}
We would like to look for matrices $(s_i)_{0\leq i \leq 8 } $ which satisfy the relation (\ref{eq5}).
The TCMs  $S_{4\otimes 4}$, $S_{8\otimes 8}$ can be expressed respectively in terms of the generalized Pauli matrices $(\sigma_i\otimes \sigma_j)_{0\leq i,j \leq3} $,
$(\sigma_i\otimes \sigma_j \otimes \sigma_k)_{0\leq i,j,k \leq 3} $ in the following way.
\begin{equation*}
S_{4\otimes 4} =\frac{1}{4}  I_4 \otimes I_4+ \frac{1}{4}\sum_{i=1}^{15} \ s_i \otimes \ s_i 
\end{equation*}
where 
$ s_1=\sigma_0 \otimes \sigma_1 $,
$ s_2=\sigma_0 \otimes \sigma_2 $, \ldots,
$ s_{13}=\sigma_3 \otimes \sigma_1 $,
$ s_{14}=\sigma_3 \otimes \sigma_2 $,
$ s_{15}=\sigma_3 \otimes \sigma_3 $.
\begin{equation*}
S_{8\otimes 8} =\frac{1}{8}  I_8 \otimes I_8 + \frac{1}{8}\sum_{i=1}^{63} \ S_i \otimes \ S_i 
\end{equation*}
where 
$ S_1=\sigma_0 \otimes \sigma_0 \otimes \sigma_1  $,...,
$ S_{63}=\sigma_3 \otimes \sigma_3 \otimes \sigma_3  $
That is to say, this generalization by the tensor products of the Pauli matrices satisfy the relation (\ref{eq5}). So we think that (\ref{eq5}) should be true for $n= 2^p$, $p\in N$, $ p \geq 2 $. For proving it, we give the 
following lemma which is the generalization of a proposition in \cite{Rakotonirina09}.\\
\begin{lemma}
If $\sum_{j=1}^m\ M_j \otimes\ N_j=\sum_{i=1}^n \ A_i \otimes \ B_i $ then\\
$\sum_{j=1}^m\ M_j \otimes\ K \otimes\ N_j=\sum_{i=1}^n \ A_i \otimes\ K\otimes\ B_i $
\end{lemma}
\begin{proof}
Let $ K=(K^{j_1}_{j_2})\in \textbf{C}^{q\times s}$, $M_{j}=M^{k_{1}}_{(j)k_2}\in \textbf{C}^{p \times r} $, $A_{i}=A^{k_{1}}_{(i)k_2}\in \textbf{C}^{p \times r} $, $N_{j}=N^{l_{1}}_{(j)l_2}\in \textbf{C}^{t \times u} $ and
$B_{i}=B^{l_{1}}_{(j)l_2}\in \textbf{C}^{t \times u} $\\
$ \sum_{j=1}^{m}\ M^{k_{1}}_{(j)k_2}  N^{l_{1}}_{(j)l_2} = \sum_{i=1}^{n}\ A^{k_{1}}_{(i)k_2}  B^{l_{1}}_{(j)l_2} $\\
$ K^{j_1}_{j_2}\sum_{j=1}^{m}\ M^{k_{1}}_{(j)k_2}  N^{l_{1}}_{(j)l_2} =   K^{j_1}_{j_2}  \sum_{i=1}^{n}\ A^{k_{1}}_{(i)k_2}  B^{l_{1}}_{(j)l_2} $\\
$\sum_{j=1}^{m}\ M^{k_{1}}_{(j)k_2} K^{j_1}_{j_2}  N^{l_{1}}_{(j)l_2} =    \sum_{i=1}^{n}\ A^{k_{1}}_{(i)k_2}  K^{j_1}_{j_2}  B^{l_{1}}_{(j)l_2} $\\
$\sum_{j=1}^{m}\ M^{k_{1}}_{(j)k_2} K^{j_1}_{j_2}  N^{l_{1}}_{(j)l_2}$ and $\sum_{i=1}^{n}\ A^{k_{1}}_{(i)k_2}  K^{j_1}_{j_2}  B^{l_{1}}_{(j)l_2} $ \\
are respectively the elements of the $k_{1}j_{1}l_{1}$ row and $k_{2}j_{2}l_{2}$ colomn \\
the $\sum_{j=1}^m\ M_j \otimes\ K \otimes\ N_j $  and $\sum_{i=1}^n \ A_i \otimes\ K\otimes\ B_i $. That is true for  any $ k_{1},j_{1},l_{1}, k_{2},j_{2},l_{2}$.
Hence, $\sum_{j=1}^m\ M_j \otimes\ K \otimes\ N_j=\sum_{i=1}^n \ A_i \otimes\ K\otimes\ B_i $
\end{proof}\\
\begin{proposition}
For any $ n\in N^{*}, n\geq2 $,
\begin{equation*}
S_{2^{n}\otimes 2^{n}} =\frac{1}{2^{n}} \sum_{i_{1},i_{2},\ldots ,i_{n}=0}^{3}(\sigma_{i_{1}} \otimes\sigma_{i_{2}} \otimes \cdots \otimes \sigma_{i_{n}}) \otimes\ (\sigma_{i_{1}} \otimes\sigma_{i_{2}} \otimes \cdots\otimes\sigma_{i_{n}})
\end{equation*}
\end{proposition}
\begin{proof}
Let us proof it by reccurence. According to the relation (\ref{eq1}). The proposition is true for $n=1$. Suppose that is true for a $ n\in N^{*}, n\geqslant 2$.
 Let us take $e_1=\begin{pmatrix}
1\\
0
            \end{pmatrix}$, $e_2=\begin{pmatrix}
0\\
1
            \end{pmatrix}$, which form a basis of the \textbf{C}-vector space $\textbf{C}^{2 \times 1}$.
It suffise us to prove 
\begin{multline*}
\frac{1}{2^{n+1}} \sum_{j_{1},j_{2}, \ldots ,j_{n+1}=0}^{3}(\sigma_{j_{1}} \otimes \cdots \otimes \sigma_{j_{n+1}}) \otimes\ (\sigma_{j_{1}} \otimes \cdots 
\otimes\sigma_{j_{n+1}}) 
(e_{\alpha_{1}} \otimes \cdots \otimes \ e_{\alpha_{n+1}}) \\
\otimes\ (e_{\beta_{1}} \otimes \cdots \otimes \ e_{\beta_{n+1}}) = 
(e_{\beta_{1}} \otimes \cdots \otimes \ e_{\beta_{n+1}}) \otimes (e_{\alpha_{1}} \otimes \cdots \otimes \ e_{\alpha_{n+1}}) 
\end{multline*} 

\begin{multline*}
\frac{1}{2^{n}} \sum_{j_{1},j_{2}, \ldots ,j_{n}=0}^{3}(\sigma_{j_{1}} \otimes \cdots \otimes \sigma_{j_{n}}) \otimes\ (\sigma_{j_{1}} \otimes \cdots \otimes\sigma_{j_{n}}) (e_{\alpha_{1}} \otimes \cdots \otimes \ e_{\alpha_{n}}) \otimes\ (e_{\beta_{1}} \otimes \cdots \otimes \ e_{\beta_{n}}) \\
=(e_{\beta_{1}} \otimes \cdots \otimes \ e_{\beta_{n}}) \otimes (e_{\alpha_{1}} \otimes \cdots \otimes \ e_{\alpha_{n}}) 
\end{multline*}
that is 

\begin{multline*}
\frac{1}{2^{n}} \sum_{j_{1},j_{2}, \ldots ,j_{n}=0}^{3}((\sigma_{j_{1}} e_{\alpha_{1}})  \otimes \cdots \otimes (\sigma_{j_{n}} e_{\alpha_{n}})) \otimes\ ((\sigma_{j_{1}} e_{\beta_{1}})\otimes \cdots \otimes (\sigma_{j_{n}} e_{\beta_{n}}))\\
=(e_{\beta_{1}} \otimes \cdots \otimes \ e_{\beta_{n}}) \otimes (e_{\alpha_{1}} \otimes \cdots \otimes \ e_{\alpha_{n}}) 
\end{multline*}
According to the lemma above
\begin{multline*}
\frac{1}{2^{n}} \frac{1}{2} \sum_{j_{1},j_{2}, \ldots ,j_{n},j_{n+1}=0}^{3}[(\sigma_{j_{1}} e_{\alpha_{1}})  \otimes \cdots \otimes (\sigma_{j_{n+1}} e_{\alpha_{n+1}})] \otimes\ [(\sigma_{j_{1}} e_{\beta_{1}})\otimes \cdots \otimes (\sigma_{j_{n+1}} e_{\beta_{n+1}})]
\\
= \frac{1}{2} \sum_{j_{n+1}=0}^{3}(e_{\beta_{1}} \otimes \cdots \otimes \ e_{\beta_{n}})\otimes(\sigma_{j_{n+1}} e_{\alpha_{n+1}})\otimes (e_{\alpha_{1}} \otimes \cdots \otimes \ e_{\alpha_{n}}) \otimes(\sigma_{j_{n+1}} e_{\beta_{n+1}})
\end{multline*}
from the relation (\ref{eq1}) 
\begin{equation*}
S_{2 \otimes 2}(\alpha_{n+1}\otimes\beta_{n+1})=
\frac{1}{2} \sum_{j_{{n+1}}=0}^{3}(\sigma_{j_{n+1}}  e_{\alpha_{n+1}}) \otimes (\sigma_{j_{n+1 }}  e_{\beta_{n+1}}) =e_{\beta_{n+1}} \otimes e_{\alpha_{n+1}}
\end{equation*}
According again to the lemma above
\begin{multline*}
\frac{1}{2} \sum_{j_{{n+1}}=0}^{3}(e_{\beta_{1}} \otimes \cdots \otimes e_{\beta_{n}}) \otimes (\sigma_{j_{n+1}}  e_{\alpha_{n+1}})
\otimes(e_{\alpha_{1}} \otimes \cdots \otimes \ e_{\alpha_{n}})\otimes (\sigma_{j_{n+1}}  e_{\beta_{n+1}}) \\
=(e_{\beta_{1}} \otimes \cdots \otimes e_{\beta_{n}}\otimes e_{\beta_{n+1}})\otimes(e_{\alpha_{1}} \otimes \cdots \otimes e_{\alpha_{n}}\otimes e_{\alpha_{n+1}})
\end{multline*}
\end{proof}
\section*{Conclusion and discussion}
We have calculated the left hand side of the relations (\ref{eq3}) and (\ref{eq4}), respectively for the $3 \times 3$ Pauli matrices of Kibler and 
the nonions in expecting to have the formula (\ref{eq5}), for $n=3$. Instead of $S_{3\otimes3}$ we have the permutation 
matrix $P$ as result. However, that makes us to think that there should be other $3 \times 3$ Pauli matrices which would satisfy the relation (\ref{eq5}). 
These $3 \times 3$ Pauli matrices should not be the normalized Gell-Mann matrices in \cite{José08}, because the $4\times4$ matrices which satisfy 
the relation (\ref{eq5}) above are not the  $4\times4$ normalized Gell-Mann matrices in \cite{José08}, even though these normalized Gell-Mann matrices satisfy (\ref{eq5}). 
The relation  (\ref{eq5}) is satisfied by the generalized Pauli matrices by tensor products of the Pauli matrices, but only for $n=2^k$. However, there is no
$3 \times 3$ matrix, formed by zeros in the diagonal which satisfy both the relations (\ref{e5}) and (\ref{e6}). Thus, the $3 \times 3$ Pauli matrices which
should satisfy (\ref{eq5}), if there exist, do not satisfy both the relations (\ref{e5}),(\ref{e6}) and (\ref{e7}) like the generalized Pauli matrices by tensor products of the Pauli
matrices.

\end{document}